\documentclass[12pt,leqno]{amsart}

\usepackage{amsmath,amsfonts,amssymb,amsthm,mathrsfs}

\usepackage[usenames]{color}

\usepackage[left=2cm,top=2.3cm,right=2cm]{geometry}

\newtheorem{thm}{Theorem}[section]
\newtheorem{prop}[thm]{Proposition}

\newtheorem{lem}[thm]{Lemma}

\theoremstyle{definition}
\newtheorem{rem}[thm]{Remark}
\newtheorem{defn}[thm]{Definition}

\DeclareMathOperator{\dima}{dim_A}
\DeclareMathOperator{\diam}{diam}
\DeclareMathOperator{\dist}{dist}

      \newcommand{\N}{{\mathbb N}}
      \newcommand{\R}{{\mathbb R}}

\begin{document}

\title[Fractional Hardy--Sobolev inequalities]
{Fractional Hardy--Sobolev type inequalities\\ for half spaces and John domains}

\author[B. Dyda]{Bart{\l}omiej Dyda}
\address[B.D.]{Faculty of Pure and Applied Mathematics, Wroc{\l}aw University of Science and Technology, ul. Wybrze\.{z}e Wyspia\'{n}skiego 27, 50-370 Wroc{\l}aw, Poland}
\email{bartlomiej.dyda@pwr.edu.pl}
 \thanks{B.D. was partially supported by National Science Centre, Poland, grant no. 2015/18/E/ST1/00239}

\author[J. Lehrb\"ack]{Juha Lehrb\"ack}   
\address[J.L.]{Department of Mathematics and Statistics, P.O. Box 35, FI-40014 University of Jyvaskyla, Finland}
\email{juha.lehrback@jyu.fi}

\author[A. V. V\"ah\"akangas]{Antti V. V\"ah\"akangas}
\address[A.V.V.]{Department of Mathematics and Statistics, P.O. Box 35, FI-40014 University of Jyvaskyla, Finland}
 \email{antti.vahakangas@iki.fi}


\begin{abstract}
As our main result we prove a variant of the fractional 
Hardy--Sobolev--Maz'ya inequality for half spaces.
This result contains a complete answer to
a recent open question by Musina and Nazarov.
In the proof we apply a new version of the 
fractional Hardy--Sobolev inequality that we establish
also for more general unbounded John domains than half spaces.
\end{abstract}

\maketitle

\markboth{\textsc{B. Dyda,  J. Lehrb\"ack and A. V. V\"ah\"akangas}}
{\textsc{Fractional Hardy--Sobolev type inequalities}}

\section{Introduction}
The main result in this note is the following  {\em fractional 
Hardy--Sobolev--Maz'ya inequality}  for functions $u\in C^\infty_0(\R^n_+)$,
where $\R^n_+$ is the upper half space of $\R^n$ with $n\ge 2$:
\begin{align}\label{eq:main_intro}
\iint_{\R^n_+\times\R^n_+} \frac{\lvert u(x)-u(y)\rvert^p}{\lvert x-y\rvert^{n+sp} } \,dy\,dx
 - \mathcal D \int_{\R^n_+} \lvert u(x)\rvert^p x_n^{-sp}\,dx 
\geq \sigma \, \biggl( \int_{\R^n_+} \lvert u(x)\rvert^q x_n^{-bq} \,dx \biggr)^{p/q}\,.
\end{align}
Here $2\leq p,q<\infty$ and $0<s<1$ are such that $sp<n$ and $p < q \leq np/(n-sp)$, and
$b=n(1/q-1/p)+s$; notice that then
\[
\frac{b}{n} = \frac{1}{q} - \frac{n-sp}{np}\quad\text{ and }\quad -bq=\frac q p(n-sp)-n\,.
\]
The constant $\sigma=\sigma(n,p,q,s)>0$ in~\eqref{eq:main_intro} is independent of $u$,
and $\mathcal D=\mathcal D(n,p,s)\ge 0$ is the optimal constant for which the left-hand side of~\eqref{eq:main_intro}
is non-negative for all $u\in C^\infty_0(\R^n_+)$; see~\eqref{eq:hardyconst} in Section~\ref{sec.application}
for an explicit expression of this constant.
By approximation, inequality~\eqref{eq:main_intro} holds for all functions in the 
associated fractional Sobolev space $\mathcal{W}_0^{s,p}(\R^n_+)$; cf.\ Theorem~\ref{thm:hsm}.

The validity of inequality~\eqref{eq:main_intro} completely solves the Open Problem~1 posed by 
Musina and Nazarov at the end of the paper~\cite{MusinaNazarov}, where actually only the
case $p=2$ was under consideration. Our results also extend the validity of~\cite[Theorem~3.1]{MusinaNazarov}
to the case $0<s<1/2$.  

When $sp=1$, the constant $\mathcal D$ equals zero, and then for 
$q=np/(n-sp)=np/(n-1)$ inequality~\eqref{eq:main_intro}
is the usual Sobolev inequality. For $sp\neq 1$ it holds that $\mathcal D>0$.
When $sp>1$, then in the special case $q = np/(n-sp)$
the validity of inequality~\eqref{eq:main_intro} 
was proved in \cite{Sloane} for $p=2$
and in \cite{MR2910984} for general $p\geq 2$; see also \cite{FMT} for similar results.
On the other hand, when $sp<1$, the validity of inequality~\eqref{eq:main_intro} seems 
to be completely new.

In the proof of inequality~\eqref{eq:main_intro}, we bring together in a novel way adaptations of some recent results related to fractional inequalities. 
We begin in Section~\ref{s.john} by
extending a fractional Riesz potential estimate from~\cite{H-SV} to the case of unbounded John domains,
including the upper half-space $\R^n_+$.
The definition and some important properties of John domains are recalled at the beginning
of that section. In Section~\ref{s.hardy-sobolev}, we establish the {\em weighted fractional Hardy--Sobolev inequality} 
\begin{equation}\label{e.weighted_intro}
\begin{split}
\int_D \int_{B(x,\tau \delta_{\partial D}(x))}  \frac{\lvert u(x)-u(y)\rvert^p}{\lvert x-y\rvert^{n+s p}}\,dy\,
\delta_{\partial D}^\beta(x)\,dx
\ge C
\bigg(\int_D \lvert u(x)\rvert^q
\delta_{\partial D}^{(q/p)(n-s p+\beta)-n}(x)\,dx\bigg)^{p/q}
\end{split}
\end{equation} 
for functions $u\in C^\infty_0(D)$, where $0<\tau<1$
and $D$ is an unbounded John domain satisfying the additional assumption that the Assouad dimension of the
boundary $\partial D$ is small enough; we use here the notation $\delta_{\partial D}(x)=\dist(x,\partial D)$. 
The proof of inequality~\eqref{e.weighted_intro} is based on the 
Riesz potential estimate from Section~\ref{s.john} and general two weight inequalities for Riesz potentials
from~\cite{Dyda0}. An important feature in inequality~\eqref{e.weighted_intro} is that,
due to the parameter $0<\tau<1$, the
inner integral in the left-hand side is taken over a ball
which is not too close to the boundary $\partial D$. This crucial fact
allows some flexibility to modify the
weight functions that are powers of the distance-to-boundary function  $\delta_{\partial D}$; cf.\ estimate~\eqref{e.shift}. 
The fractional Hardy--Sobolev--Maz'ya inequality~\eqref{eq:main_intro} is then proved in Section~\ref{sec.application}, 
relying on the Hardy--Sobolev inequality~\eqref{e.weighted_intro} and a sharp fractional
Hardy inequality with a reminder term from~\cite[Theorem~1.2]{FrankSeiringer}.
Section~\ref{sec.application} also contains discussion related to the
space $\mathcal{W}_0^{s,p}(\R^n_+)$ and the approximation argument that
allows us to extend the validity of inequality~\eqref{eq:main_intro} for the functions
belonging to this space.

\subsection*{Notation}\label{s.notation}

Throughout this note, we work in the $n$-dimensional Euclidean space $\R^n$, with $n\ge 2$.
We write $\R^n_+=\R^{n-1}\times (0,\infty)$, 
and denote by $C^\infty_0(\R^n_+)$ the space of smooth functions 
whose support is a compact set in $\R^n_+$.
The open ball centered at $x\in \R^n$ and with radius $r>0$ is denoted $B(x,r)$.
When $E\neq\emptyset$ is a set in $\R^n$, 
the Euclidean distance from $x\in\R^n$ to $E$ is written as $\dist(x,E)=\delta_E(x)$,
the diameter of $E$ is $\diam(E)$, and
we write $\chi_E$ for the characteristic function of $E$; that is, 
$\chi_E(x)=1$ if $x\in E$ and $\chi_E(x)=0$ if $x\notin E$.
In addition, $\overline{E}$ denotes the closure of $E$. 
The Lebesgue $n$-measure of a measurable set $E\subset \R^n$ is denoted by $\vert E\vert$,
and if $0<\vert E\vert<\infty$ and $u$ is an integrable function on $E$, we use the notation
\[
u_E=\frac{1}{\lvert E \rvert} \int_{E}u(y)\,dy\,. 
\]

The letter $C$ is used to denote positive constants whose values are
not necessarily the same at each occurrence.
We also write $C=C(\ast,\dotsb,\ast)$ to indicate that the constant $C$
depends (at most) on the quantities appearing
in the parentheses.

\section{A fractional potential estimate on John domains}\label{s.john}

In this section we prove Theorem \ref{t.repr}, which provides a fractional
potential estimate for unbounded John domains;
recall that a domain is an open and connected set.
Following \cite{MR1246886}, we will first define John domains 
in such a way that unbounded domains are allowed. 
Several equivalent definitions for John domains can be found in \cite{MR1246886}.
 When $D\subset\R^n$ is a domain and $x_1,x_2\in D$, we say that a curve $\gamma\colon [0,\ell]\to D$
joins $x_1$ to $x_2$ if $\gamma(0)=x_1$ and $\gamma(\ell)=x_2$.

\begin{defn}\label{sjohn}
A domain $D\subsetneq\R^n$, with $n\ge 2$, is a {\em $c$-John domain}, for $c\ge 1$, if
each pair of points $x_1,x_2\in D$ can be joined
by a rectifiable arc length parametrized curve $\gamma\colon [0,\ell]\to D$ satisfying
$\dist(\gamma(t),\partial D)\ge \min\{t,\ell -t\}/c$ for every $t\in [0,\ell]$.
\end{defn}

It is clear that for example the 
half-space $\R^n_+=\R^{n-1}\times (0,\infty)$ is an unbounded John domain,
but it is also easy to come up with more irregular
examples, since the class of John domains is quite flexible.
For instance, the unbounded domain whose boundary is the usual von Koch snowflake
curve is an unbounded John domain in $\R^2$.

The next lemma recalls a useful property 
which can actually be used to characterize bounded John domains.
See \cite[Theorem 3.6]{MR1246886} for more details and a proof of this result.

\begin{lem}\label{t.equi}
Assume that $D\subset\R^n$ is a bounded $c_1$-John domain, $n\ge 2$. 
 Then there is a point $x_0\in D$ such that each
 $x\in D$ can be joined to $x_0$ by a rectifiable arc length parametrized curve
 $\gamma\colon[0,\ell]\to D$ satisfying 
$\dist(\gamma(t),\partial D)\ge t/(4c_1^2)$ for every $t\in [0,\ell]$. 
\end{lem}

The point $x_0$ appearing in Lemma \ref{t.equi} is called a {\em John center of $D$}.
The following engulfing property of John domains can be found in \cite[Theorem 4.6]{MR1246886}.

\begin{lem}\label{t.engulfing}
A $c$-John domain $D\subsetneq\R^n$ can be written as the union
of $c_1$-John domains $D_1,D_2,\ldots$, where $c_1=c_1(c,n)$
and $\overline{D_i}$ is compact in $D_{i+1}$ 
for each $i=1,2,\ldots$.
\end{lem}

We now turn to the potential estimate, which is given in terms of Riesz potentials.
Recall that the \emph{Riesz potential} $\mathcal{I}_\alpha(f)$ of a measurable function $f\colon \R^n\to [0,\infty]$, 
for $0<\alpha<n$, is defined as
\[
\mathcal{I}_\alpha(f)(x) = \int_{\R^n} \frac{f(y)}{\lvert x-y\rvert^{n-\alpha}} \,dy\,,\qquad x\in \mathbb{R}^n\,.
\]

\begin{thm}\label{t.repr}
Assume that $D\subsetneq \R^n$ is an unbounded $c$-John domain,
and let $0<\tau,s<1$ and $1\le p<\infty$.
Then there is a constant $C=C(\tau,n,c,s,p)>0$ such that 
the inequality
\begin{align*}
\vert u(x)\vert &\le C
\int_{D}\frac{g(y)}{\vert x-y\vert ^{n-s}}\,dy =C\,\mathcal{I}_{s}(\chi_D g)(x)
\end{align*}
holds whenever  $u\in \bigcup_{1\le r<\infty}L^r(D)$ and
$x\in D$ is Lebesgue point of $u$, where we have denoted
\begin{equation}\label{e.finite}
g(y):=\biggl(\int_{B(y,\tau\delta_{\partial D}(y))}\frac{\vert u(y)-u(z)\vert^p}{\vert y-z\vert ^{n+s p}}\,dz\biggr)^{1/p}
\end{equation}
for every $y\in D$.
\end{thm}

For the proof of Theorem \ref{t.repr}, we first need a
a fractional potential estimate for bounded John domains, which 
is stated in Proposition \ref{t.bounded_repr} below.
For a simple proof of Proposition \ref{t.bounded_repr}, we refer to
the proof of \cite[Theorem 4.10]{H-SV};
see formula (4.13) therein. 
For our purposes, we actually need to track the constants a bit
more carefully than what is done in \cite{H-SV},
but an inspection of the proof in \cite{H-SV} 
shows that the constants depend on the $c_1$-John domain $D$
only through $c_1$; we omit further details.
We also remark that while the statement of \cite[Theorem 4.10]{H-SV} contains the assumption 
$sp<n$, this is not needed for \cite[formula (4.13)]{H-SV} to hold.

\begin{prop}\label{t.bounded_repr}
Assume that $D\subset \R^n$ is a bounded $c_1$-John domain,
and let $0<\tau,s<1$ and $1\le p<\infty$. 
Let $x_0\in D$ be a John center of $D$ as in Lemma \ref{t.equi}, 
let $M>2/\tau$, and denote 
\[
B=B\biggl(x_0,\frac{\delta_{\partial D}(x_0)}{16Mc^2_1}\biggr)\,.
\]
Then there is a constant $C=C(M,n,c_1,s,p)>0$ such that
\begin{align*}
\vert u(x)-u_{B}\vert &\le C
\int_{D}\frac{g(y)}{\vert x-y\vert ^{n-s}}\,dy
\end{align*}
whenever $u\in L^1_{\textup{loc}}(D)$, $x\in D$ is a
Lebesgue point of $u$, and $g$ is as in~\eqref{e.finite}
with respect to the bounded $c_1$-John domain $D$.
\end{prop}

We are now ready for the proof of Theorem \ref{t.repr}.

\begin{proof}[Proof of Theorem \ref{t.repr}]
Assume that $u\in L^r(D)$ for some $1\le r<\infty$ and choose $M=3/\tau$.
By Lemma~\ref{t.engulfing}, there are bounded $c_1$-John domains $D_i$ with $c_1=c_1(c,n)\ge 1$ such that
\begin{equation*}
D_i\subset \overline{D_i}\subset D_{i+1}\,,\quad \text{ for all }i=1,2,\dots\,,
\end{equation*}
and
$D=\bigcup_{i=1}^{\infty}D_i$. 
Let $x_i\in D_i$ be a John center of $D_i$ given by Lemma \ref{t.equi}, and write
\[
B_i:=B\biggl(x_i,\frac{\delta_{\partial D_i}(x_i)}{16Mc^2_1}\biggr)\subset D_i\subset D\,.
\]
By Lemma \ref{t.equi} we have $\delta_{\partial D_i}(x_i) \ge (12c_1^2)^{-1}\diam(D_i)$.   
Observe that the numbers $\diam(D_i)$ converge to $\infty$ as $i\to \infty$, and thus 
$\lim_{i\to\infty}\lvert B_i\rvert  = \infty$. In particular, 
by H\"older's inequality, 
\[
\lvert u_{B_i} \rvert \le \frac{1}{\lvert B_i\rvert}\int_{B_i}\lvert u(x)\rvert\,dx\le \frac{\lVert u\rVert_{L^r(D)}}{\lvert B_i\rvert^{1/r}}
\xrightarrow{i\to\infty}0\,.
\]
Let us denote by $g_i$ the function defined as in \eqref{e.finite},
but with respect to the bounded $c_1$-John domain $D_i$, 
and let $x\in D$  be a Lebesgue point of $u$. Since $x\in D_i$  
for all sufficiently large indices $i$  
and $u\rvert_{ D_i} \in  L^1_{\textup{loc}}(D_i)$, 
we find, by an application of Proposition \ref{t.bounded_repr} 
and monotone convergence, 
that
\begin{align*}
\lvert u(x)\rvert &=\lim_{i\to\infty} \lvert u(x)-u_{B_i}\rvert 
\le C(M,n,c_1,s,p) \limsup_{i\to\infty} \int_{D_i}\frac{g_i(y)}{\lvert x-y\rvert^{n-s}}\,dy\\
&\le C(M,n,c_1,s,p) \int_D \frac{g(y)}{\lvert x-y\rvert^{n-s}}\,dy\,.
\end{align*}
This concludes the proof of the theorem.
\end{proof}

\section{Weighted fractional Hardy--Sobolev inequalities}\label{s.hardy-sobolev}

In this section we establish  
weighted fractional inequalities of the general form
\begin{equation}\label{e.weighted}
\begin{split}
\bigg(\int_D \lvert u(x)\rvert^q
\delta_{\partial D}^{(q/p)(n-s p+\beta)-n}(x)\,dx\bigg)^{p/q}
\le C\int_D \int_{B(x,\tau \delta_{\partial D}(x))}  \frac{\lvert u(x)-u(y)\rvert^p}{\lvert x-y\rvert^{n+s p}}\,dy\,
\delta_{\partial D}^\beta(x)\,dx\,,
\end{split}
\end{equation} 
where $u\in L^r(D)$ for some $1\le r<\infty$ and $D\subsetneq\R^n$ is an unbounded John domain satisfying the dimensional condition~\eqref{e.dim_assumption} below. Recall that
we write $\delta_{\partial D}(x)=\dist(x,\partial D)$. 
As was already mentioned in the Introduction, an
important feature here is that we obtain inequality~\eqref{e.weighted} with 
a parameter $0<\tau<1$. This allows us to use in applications of~\eqref{e.weighted} estimates of the type
\begin{equation}\label{e.shift}
\delta_{\partial D}^\beta(x)\le C \delta_{\partial D}^{\beta_1}(y)\delta_{\partial D}^{\beta_2}(x)
\end{equation}
for $x\in D$ and $y\in B(x,\tau \delta_{\partial D}(x))$, where and $\beta_1+\beta_2=\beta$
and $C=C(\tau,\beta_1,\beta_2)$.

When $E\subset \R^n$, the \emph{Assouad dimension} denoted 
$\dima(E)$ is the infimum of exponents $\alpha\ge 0$ for which 
there is a constant $C \ge 1$ such that for each $x\in E$ and every $0<r<R$, the set $E\cap B(x,R)$ 
can be covered by at most $C(r/R)^{-\alpha}$ balls of radius $r$. 
For example, the Assouad dimension of the boudary of the half-space $\R^n_+$ is $\dima(\partial\R^n_+)=n-1$,
and more generally, if $E\subset \R^n$ is an $m$-dimensional subspace, then $\dima(E)=m$.
See e.g.~\cite{Luukkainen} for more details, properties, and examples 
related to the Assouad dimension.

The following Theorem \ref{t.hardy} is 
a partial generalization of \cite[Theorem 1]{Dyda1},
where a weighted fractional Hardy-type inequality (the case $q=p$) is addressed,
and it extends \cite[Theorem 5.2]{H-SV2}, where a fractional Sobolev inequality 
is obtained in the case when $\beta=0$ and $q=np/(n-s p)$.
Theorem \ref{t.hardy} is an improved version of the recent metric
space result \cite[Theorem 5.3]{Dyda0}, where 
all the integrals were taken over the whole space.

\begin{thm}\label{t.hardy}
Assume that $D\subsetneq \R^n$ is an unbounded $c$-John domain and that $0<s<1$, 
\[
1<p\le q\le \frac{np}{n-s p}<\infty\,,
\]
and $\beta\in\R$ are such that
\begin{equation}\label{e.dim_assumption}
\dima(\partial D) < \min \biggl\{ \frac{q}{p}(n-s p+\beta) \, , \, n- \frac{\beta}{p-1} \biggr\}\,.
\end{equation}
In addition, let $\tau\in (0,1)$.
Then there is a constant $C=C(\beta,\tau,n,c,s,p,q)>0$ such that inequality~\eqref{e.weighted}
holds for all $u\in \bigcup_{1\le r<\infty} L^r(D)$.
\end{thm}

\begin{proof}
Fix a function $u\in \bigcup_{1\le r<\infty} L^r(D)$
and write, as in~\eqref{e.finite}, for every $y\in D$, 
\begin{equation*}
g(y)=\biggl(\int_{B(y,\tau\delta_{\partial D}(y))}\frac{\vert u(y)-u(z)\vert^p}{\vert y-z\vert ^{n+s p}}\,dz\biggr)^{1/p}\,.
\end{equation*}
Also denote, for every $x\in \R^n\setminus \partial D$,
\[
w(x) = \delta_{\partial D}^{(q/p)(n-s p+\beta)-n}(x)\,. 
\] 
We remark that $w$ is defined and positive almost everywhere in $\R^n$. Indeed, notice first that $\dima(\partial D)<n$ by the assumption \eqref{e.dim_assumption}, and thus $\lvert \partial D\rvert=0$; we refer to \cite[Remark 3.2]{Dyda0}.

By Theorem \ref{t.repr}, there is a constant $C=C(\tau,n,c,s,p)>0$ such that
inequality
\begin{align*}
\vert u(x)\vert^q w(x)&\le C\,\mathcal{I}_{s}(\chi_D g)(x)^q w(x)
\end{align*}
holds for every Lebesgue point $x\in D$ of $u$.  In particular, since almost every point $x\in D$ is a Lebesgue point of $u$, we obtain that
\begin{align*}
\bigg(\int_{D}\vert u(x)\vert^q w(x)\,dx\biggr)^{p/q} &\le  C\bigg(\int_{D}\mathcal{I}_{s}(\chi_D g)(x)^q w(x)\,dx\biggr)^{p/q}\\
&\le C\bigg(\int_{\R^n}\mathcal{I}_{s}(\chi_D g)(x)^q w(x)\,dx\biggr)^{p/q}\,.
\end{align*}
Next we apply~\cite[Theorem 4.1]{Dyda0}, 
which yields two weight inequalities for the Riesz potentials, where the weights are powers of the distance function $\delta_{\partial D}$. In~\cite{Dyda0} the result is formulated in a general metric space,
but it is straightforward to see that in $\R^n$ the dimensional condition in \cite[Theorem 4.1]{Dyda0} coincides with~\eqref{e.dim_assumption}.
We remark that the proof of \cite[Theorem 4.1]{Dyda0}
is based on the Muckenhoupt $A_p$-properties of the powers of $\delta_{\partial D}$ and
general $A_p$-weighted inequalities; the Euclidean space versions of the latter are 
originally due to P{\'e}rez~\cite{Perez1990}.
From~\cite[Theorem 4.1]{Dyda0} it follows that  
\begin{align*}
\bigg(\int_{\R^n}\mathcal{I}_{s}(\chi_D g)(x)^q w(x)\,dx\bigg)^{p/q}
&\le C\int_{\R^n} \chi_D(y) g(y)^p\,\delta_{\partial D}^{\beta}(y)\,dy
\\&=C\int_D \int_{B(y,\tau\delta_{\partial D}(y))}\frac{\vert u(y)-u(z)\vert^p}{\vert y-z\vert ^{n+s p}}\,dz\,\delta_{\partial D}^{\beta}(y)\,dy\,.
\end{align*}
Here the constant $C>0$ is independent of $u$ and $g$, and so the desired inequality \eqref{e.weighted} follows by combining the  two estimates above.
\end{proof}

\begin{rem}\label{r.admissible}
In the case $D=\R^n_+$ we have $\dima(\partial\R^n_+)=n-1$. 
Then the bounds in~\eqref{e.dim_assumption}
are equivalent to
\[
\frac p q (n-1) - n +sp <\beta <p-1.
\]
From this we see that the lower bound for $\beta$ is strictly decreasing in terms of $q$.
For $q=p$ the lower bound is $sp-1$ and for $q=np/(n-sp)$ the lower bound is
$sp/n-1$. In particular, the value $\beta=sp-1$, which will be used in the following
Section~\ref{sec.application} while proving our main inequality~\eqref{eq:main_intro}, 
is allowed in~\eqref{e.weighted} for $D=\R^n_+$ whenever 
$0<s<1$ and $1< p<q\le np/(n-sp) < \infty$. 

Let us however point out that we do not know if the above bounds for $\beta$ are
optimal in $\R^n_+$ or in more general unbounded $c$-John domains; in particular, the necessity
of the upper bound $\beta<p-1$ is questionable. 
\end{rem}

\begin{rem}
Assume that $D\subset\R^n$ is a bounded $c_1$-John domain such that~\eqref{e.dim_assumption} holds, 
where $s,p,q,\beta$ are as in Theorem~\ref{t.hardy}, and let $u\in L^1_{\textup{loc}}(D)$. 
For each $y\in D$, we write
\[
g(y)=\biggl(\int_{B(y,\tau\delta_{\partial D}(y))}\frac{\vert u(y)-u(z)\vert^p}{\vert y-z\vert ^{n+s p}}\,dz\biggr)^{1/p}\,.
\]
Then it follows from Proposition~\ref{t.bounded_repr} that
\begin{align*}
\lvert u(x)-u_B\rvert^q w(x)&\le C\,\mathcal{I}_{s}(\chi_D g)(x)^q w(x)
\end{align*}
for every Lebesgue point $x\in D$ of $u$, where $B$ is as in Proposition~\ref{t.bounded_repr}
and $w$ is as in the proof of Theorem~\ref{t.hardy}. We can then repeat the rest of the proof of
Theorem~\ref{t.hardy}, and conclude that
\begin{equation*}\label{e.weighted_bdd}
\begin{split}
\biggl(\int_D \lvert u(x)-u_B\rvert^q
\delta_{\partial D}^{(q/p)(n-s p+\beta)-n}(x)\,dx\biggr)^{p/q}
\le C \int_D \int_{B(x,\tau \delta_{\partial D}(x))}  \frac{\lvert u(x)-u(y)\rvert^p}{\lvert x-y\rvert^{n+s p}}\,dy\,
\delta_{\partial D}^\beta(x)\,dx\,.
\end{split}
\end{equation*}
\end{rem}

\begin{rem}
We note that both Theorem~\ref{t.repr} and Theorem~\ref{t.hardy} hold for every $u\in L^1_{\mathrm{loc}}(D)$ sastisfying $u_{B_i} \to 0$ whenever $B_i\subset D$ is a~sequence of balls with $\diam(B_i)\to\infty$.
For example, it is enough that $u\in L^1_{\mathrm{loc}}(D)$ and $u(x)\to 0$ as $|x|\to\infty$.
\end{rem}
  
\section{Fractional Hardy--Sobolev--Maz'ya inequality on half spaces}\label{sec.application}

We are now prepared to prove the fractional Hardy--Sobolev--Maz'ya inequality~\eqref{eq:main_intro}
in the half space $\R^n_+$. As we will see, this inequality, which is reformulated in Theorem~\ref{thm:hsm} below, 
is a rather immediate consequence of Theorem~\ref{t.hardy}
and the fractional Hardy inequality with the best constant $\mathcal D$
and a~remainder term, \cite[Theorem~1.2]{FrankSeiringer}.
This constant $\mathcal D$ has the explicit form 
\begin{equation}\label{eq:hardyconst}
\mathcal D = \mathcal D(n,p,s) = 2 \pi^{\frac{n-1}2} \frac{\Gamma(\frac{1+sp}2)}{\Gamma(\frac{n+sp}2)}
\int_0^1 \left| 1 - r^{(sp-1)/p} \right|^p \frac{dr}{(1-r)^{1+sp}}\,,
\end{equation}
where $\Gamma$ denotes the usual gamma function.
In particular, $\mathcal D$ is the largest number for which the  left-hand side of the
Hardy--Sobolev--Maz'ya inequality \eqref{eq:main_intro}
is non-negative 
for every $u\in C^\infty_0(\R^n_+)$; 
see~\cite[Theorem 1.1]{FrankSeiringer}. 
We also refer to \cite{KBBD-bc, FLS, FrSe1, MR2659764} for more results concerning fractional
Hardy inequalities with best constants.

We actually prove inequality~\eqref{eq:main_intro} in Theorem~\ref{thm:hsm} for
functions belonging to space $\mathcal{W}_0^{s,p}(\R^n_+)$, which is defined as follows. 
When $1\le p < \infty$ and $0<s<1$, 
the fractional Sobolev seminorm $\lvert u\rvert_{W^{s,p}(\R^n_+)}$ 
of a measurable function $u\colon \R^n_+\to \R$ is
\[
\lvert u \rvert_{W^{s,p}(\R^n_+)} = \biggl(\iint_{\R^n_+\times \R^n_+}\frac{\lvert u(x)-u(y)\rvert^p}{\lvert x-y\rvert^{n+s p}}\,
dy\,dx\,\biggr)^{1/p}\,.
\]
The space $\mathcal{W}_0^{s,p}(\R^n_+)$ 
is then the completion of $C_0^\infty(\R^n_+)$ with respect to the seminorm   
$\lvert\,\cdot\,\rvert_{W^{s,p}(\R^n_+)}$.

\begin{rem}\label{r.indentification}
Assume that $1\le p < n/s$, where $n\ge 2$, $1\le p<\infty$ and $0<s<1$. 
Then the space $\mathcal{W}_0^{s,p}(\R^n_+)$ 
can be identified as a subspace of $L^{np/(n-sp)}(\R^n_+)$ using the following reasoning.
First, by the Sobolev Embedding Theorem \cite[Theorem 5.2]{H-SV2}, there exists a constant $C>0$ such that 
$\lVert u\rVert_{L^{np/(n-sp)}(\R^n_+)} \le C \lvert u\rvert_{W^{s,p}(\R^n_+)}$ for all $u\in C^\infty_0(\R^n_+)$.
Therefore, if $(u_j)_{j\in\N}\subset C^\infty_0(\R^n_+)$ is a Cauchy sequence with respect to the seminorm $\lvert\,\cdot\,\rvert_{W^{s,p}(\R^n_+)}$, then there exists
$u\in L^{np/(n-sp)}(\R^n_+)$ such that $\lim_{j\to \infty}\lVert u_j-u\rVert_{L^{np/(n-sp)}(\R^n_+)}=0$.
A straightforward adaptation of \cite[Proposition 7]{DydaKassman} then shows that 
$\lim_{j\to\infty} \lvert u-u_j\rvert_{W^{s,p}(\R^n_+)}=0$.  
\end{rem}

\begin{thm}\label{thm:hsm}
Let $n\ge 2$ and assume that
$2\leq p,q<\infty$ and $0<s<1$ are such that $sp<n$ and $p < q \leq np/(n-sp)$, and
write $b=n(1/q-1/p)+s$. 
Then there is a constant $\sigma=\sigma(n,p,q,s)>0$ such that
\begin{align*}
\iint_{\R^n_+\times\R^n_+} \frac{\lvert u(x)-u(y)\rvert^p}{|x-y\rvert^{n+sp} }\,dy\,dx
 - \mathcal D\int_{\R^n_+} \lvert u(x)\rvert^p x_n^{-sp} \,dx 
\geq \sigma \biggl( \int_{\R^n_+} \lvert u(x)\rvert^q x_n^{-bq} \,dx \biggr)^{p/q}
\end{align*}
for all $u\in \mathcal{W}_0^{s,p}(\R^n_+)$, 
where the constant $\mathcal D=\mathcal D(n,p,s)$ is as in \eqref{eq:hardyconst}.
\end{thm}

\begin{proof}
The proof follows the ideas presented in \cite[Section 2]{MR2910984}, but instead of using the 
Sobolev inequality as in~\cite{MR2910984}, we will use the more general inequality \eqref{e.weighted}.

We consider first the case $u\in C^\infty_0(\R^n_+)$.
Our starting point is the inequality
\begin{equation}
\label{eq:gsr}
\iint_{\R^n_+\times\R^n_+} \frac{\lvert u(x)-u(y)\rvert^p}{\lvert x-y\rvert^{n+sp} }\,dy\,dx
 - \mathcal D \int_{\R^n_+} \lvert u(x)\rvert^p x_n^{-sp} \,dx 
\geq c_p J[v] \,,
\end{equation}
where $c_p>0$ is an explicit constant (for $p=2$, \eqref{eq:gsr} is an identity with $c_2=1$),
\[
J[v] := \iint_{\R^n_+\times\R^n_+} \frac{\lvert v(x)-v(y)\rvert^p}{\lvert x-y\rvert^{n+sp} } (x_n y_n)^{(sp-1)/2}\,dy\,dx \,,
\]
and $v(x):=x_n^{-(sp-1)/p} u(x)$ for each $x\in \R^n_+$. Notice that $v\in C^\infty_0(\R^n_+)$
and $x_n=\delta_{\partial\R^n_+}(x)$ if $x\in \R^n_+$. Inequality \eqref{eq:gsr} was derived in \cite[Theorem~1.2]{FrankSeiringer}, using the `ground state representation' method from \cite{FrSe1}.

We apply Theorem~\ref{t.hardy} for $D=\R^n_+$, $\beta=sp-1$ and a~fixed $0<\tau<1$;  
recall here Remark~\ref{r.admissible}. Then we use estimate~\eqref{e.shift} with $\beta_1=\beta_2=\beta/2$, and obtain that
\begin{align*}
\biggl( \int_{\R^n_+} \lvert v(x)\rvert^q
    x_n^{(q/p)(n-sp+\beta)-n}(x)\,dx \biggr)^{p/q}
& \le C \int_{\R^n_+} \int_{B(x,\tau x_n)} \frac{\lvert v(x)-v(y)\rvert^p}{\lvert x-y\rvert^{n+sp}}\,dy\,
x_n^\beta \,dx \\
& \le
 C \int_{\R^n_+} \int_{B(x,\tau x_n)} \frac{\lvert v(x)-v(y)\rvert^p}{\lvert x-y\rvert^{n+sp}} y_n^{\beta/2} \,dy\,
x_n^{\beta/2} \,dx
\\& \le C J[v]. 
\end{align*}
Combining the above inequality with \eqref{eq:gsr} and the fact that
\[ 
\lvert v(x)\rvert^q x_n^{(q/p)(n-sp+\beta)-n}
= \lvert u(x)\rvert^q x_n^{(q/p)(n-sp)-n}
= \lvert u(x)\rvert^q x_n^{-bq}
\] 
proves the claim
for functions $u\in C^\infty_0(\R^n_+)$.

In the general case $u\in \mathcal{W}_0^{s,p}(\R^n_+)\subset L^{np/(n-sp)}(\R^n_+)$,
it suffices to consider a sequence  $(u_j)_{j\in\N}$ of $C^\infty_0(\R^n_+)$ functions, which
is Cauchy with respect to the seminorm $\lvert\,\cdot\,\rvert_{W^{s,p}(\R^n_+)}$ and which converges to $u$ in $L^{np/(n-sp)}(\R^n_+)$.  
Then $\lim_{j\to\infty}\lvert u-u_j\rvert_{W^{s,p}(\R^n_+)}=0$; cf.\ Remark \ref{r.indentification}.
By taking a subsequence, if necessary, we may also assume
that $\lim_{j\to \infty} u_j(x)=u(x)$ for almost every $x\in \R^n_+$.
By Fatou's lemma, and the already proved
inequality \eqref{eq:main_intro} for $C^\infty_0(\R^n_+)$ functions,
\begin{align*}
&\sigma \, \biggl( \int_{\R^n_+}\lvert u(x)\lvert^q x_n^{-bq} \,dx \biggr)^{p/q}\le \liminf_{j\to \infty}
\sigma \, \biggl( \int_{\R^n_+} \lvert u_j(x)\lvert^q x_n^{-bq} \,dx \biggr)^{p/q}\\
&\qquad \qquad \le \liminf_{j\to\infty} \biggl(\iint_{\R^n_+\times\R^n_+} \frac{\lvert u_j(x)-u_j(y)\lvert^p}{\lvert x-y\lvert^{n+sp} }\,dy\,dx
 - \mathcal D \int_{\R^n_+} \lvert u_j(x)\lvert^p x_n^{-sp}\,dx\biggr)\,.
\end{align*}
When $\mathcal D =\mathcal D(n,p,s)\not=0$, we have $sp\not=1$, and therefore, by the fractional Hardy inequality in \cite[Theorem 1.1]{FrankSeiringer},
\[
\biggl(\mathcal D\int_{\R^n_+}\lvert u_j(x)\lvert^p x_n^{-sp}\,dx\biggr)^{1/p}+ \biggl(\mathcal D\int_{\R^n_+}\lvert u(x)\lvert^p x_n^{-sp}\,dx\biggr)^{1/p}
\le \lvert u_j\rvert_{W^{s,p}(\R^n_+)} + \lvert u\rvert_{W^{s,p}(\R^n_+)} <\infty\,,
\]
and furthermore
\begin{align*}
\biggl\lvert \biggl(\mathcal D&\int_{\R^n_+}\lvert u_j(x)\lvert^p x_n^{-sp}\,dx\biggr)^{1/p}- \biggl(\mathcal D\int_{\R^n_+}\lvert u(x)\lvert^p x_n^{-sp}\,dx\biggr)^{1/p}\biggr\rvert
\\&\le  \biggl(\mathcal D\int_{\R^n_+}\lvert u(x)-u_j(x)\lvert^p x_n^{-sp}\,dx\biggr)^{1/p}
\le  \lvert u-u_j\rvert_{W^{s,p}(\R^n_+)}\xrightarrow{j\to\infty} 0\,.
\end{align*}
Since $\lim_{j\to \infty}\lvert u-u_j\rvert_{W^{s,p}(\R^n_+)}=0$, we  find that 
$\lim_{j\to \infty} \lvert u_j\rvert_{W^{s,p}(\R^n_+)}=\lvert u\rvert_{W^{s,p}(\R^n_+)}$. Hence,
\begin{align*}
&\liminf_{j\to\infty} \biggl(\iint_{\R^n_+\times\R^n_+} \frac{\lvert u_j(x)-u_j(y)\lvert^p}{\lvert x-y\lvert^{n+sp} }\,dy\,dx
 - \mathcal D\int_{\R^n_+} \lvert u_j(x)\lvert^p x_n^{-sp}\,dx\biggr)\\
 &\qquad\qquad\qquad\qquad=\iint_{\R^n_+\times\R^n_+} \frac{\lvert u(x)-u(y)\lvert^p}{\lvert x-y\lvert^{n+sp} }\,dy\,dx
 - \mathcal D \int_{\R^n_+} \lvert u(x)\rvert^p x_n^{-sp}\,dx\,.
\end{align*}
The claim follows from the above estimates. 
\end{proof}


\begin{thebibliography}{10}

\bibitem{KBBD-bc}
K.~Bogdan and B.~Dyda.
\newblock The best constant in a fractional {H}ardy inequality.
\newblock {\em Math. Nachr.}, 284(5-6):629--638, 2011.

\bibitem{MR2910984}
B.~Dyda and R.~L. Frank.
\newblock Fractional {H}ardy-{S}obolev-{M}az'ya inequality for domains.
\newblock {\em Studia Math.}, 208(2):151--166, 2012.

\bibitem{Dyda0}
B.~Dyda, L.~Ihnatsyeva, J.~Lehrb\"ack, H.~Tuominen, and A.~V. V\"ah\"akangas.
\newblock Muckenhoupt {$A_p$}-properties of distance functions and applications
  to {H}ardy--{S}obolev -type inequalities.
\newblock arXiv:1705.01360 [math.CA], 2017.

\bibitem{DydaKassman}
B.~Dyda and M.~Kassmann.
\newblock {F}unction spaces and extension results for nonlocal {D}irichlet
  problems.
\newblock arXiv:1612.01628 [math.AP], 2016.

\bibitem{Dyda1}
B.~Dyda and A.~V. V\"ah\"akangas.
\newblock A framework for fractional {H}ardy inequalities.
\newblock {\em Ann. Acad. Sci. Fenn. Math.}, 39(2):675--689, 2014.

\bibitem{FMT}
S.~Filippas, L.~Moschini, and A.~Tertikas.
\newblock Sharp trace {H}ardy-{S}obolev-{M}az'ya inequalities and the
  fractional {L}aplacian.
\newblock {\em Arch. Ration. Mech. Anal.}, 208(1):109--161, 2013.

\bibitem{FLS}
R.~L. Frank, E.~H. Lieb, and R.~Seiringer.
\newblock Hardy-{L}ieb-{T}hirring inequalities for fractional {S}chr\"odinger
  operators.
\newblock {\em J. Amer. Math. Soc.}, 21(4):925--950, 2008.

\bibitem{FrSe1}
R.~L. Frank and R.~Seiringer.
\newblock Non-linear ground state representations and sharp {H}ardy
  inequalities.
\newblock {\em J. Funct. Anal.}, 255(12):3407--3430, 2008.

\bibitem{FrankSeiringer}
R.~L. Frank and R.~Seiringer.
\newblock Sharp fractional {H}ardy inequalities in half-spaces.
\newblock In {\em Around the research of {V}ladimir {M}az'ya. {I}}, volume~11
  of {\em Int. Math. Ser. (N. Y.)}, pages 161--167. Springer, New York, 2010.

\bibitem{H-SV}
R.~Hurri-Syrj\"anen and A.~V. V\"ah\"akangas.
\newblock On fractional {P}oincar\'e inequalities.
\newblock {\em J. Anal. Math.}, 120:85--104, 2013.

\bibitem{H-SV2}
R.~Hurri-Syrj\"anen and A.~V. V\"ah\"akangas.
\newblock Fractional {S}obolev-{P}oincar\'e and fractional {H}ardy inequalities
  in unbounded {J}ohn domains.
\newblock {\em Mathematika}, 61(2):385--401, 2015.

\bibitem{MR2659764}
M.~Loss and C.~Sloane.
\newblock Hardy inequalities for fractional integrals on general domains.
\newblock {\em J. Funct. Anal.}, 259(6):1369--1379, 2010.

\bibitem{Luukkainen}
J.~Luukkainen.
\newblock Assouad dimension: antifractal metrization, porous sets, and
  homogeneous measures.
\newblock {\em J. Korean Math. Soc.}, 35(1):23--76, 1998.


\bibitem{MusinaNazarov}
R.~Musina and A.~I.~Nazarov.
\newblock {F}ractional {H}ardy--{S}obolev inequalities on half spaces.
\newblock arXiv:1707.02710 [math.AP], 2017.

\bibitem{Perez1990}
C.~P{\'e}rez.
\newblock Two weighted norm inequalities for {R}iesz potentials and uniform
 {$L\sp p$}-weighted {S}obolev inequalities.
\newblock {\em Indiana Univ. Math. J.}, 39(1):31--44, 1990.


\bibitem{Sloane}
C.~A. Sloane.
\newblock A fractional {H}ardy-{S}obolev-{M}az'ya inequality on the upper
  halfspace.
\newblock {\em Proc. Amer. Math. Soc.}, 139(11):4003--4016, 2011.

\bibitem{MR1246886}
J.~V\"ais\"al\"a.
\newblock Exhaustions of {J}ohn domains.
\newblock {\em Ann. Acad. Sci. Fenn. Ser. A I Math.}, 19(1):47--57, 1994.

\end{thebibliography}

%

\end{document}